\documentclass[12pt]{article}
\usepackage{amsfonts}
\usepackage{amsthm}
\usepackage{latexsym}
\usepackage{amssymb}
\usepackage{amstext}
\usepackage{amsmath}

\usepackage{graphicx}
\usepackage{amssymb}
\usepackage{epstopdf}
\usepackage{tikz}
\usepackage{graphicx}
\usepackage{amssymb}

\newtheorem{theorem}{Theorem}

\newtheorem{lem}{Lemma}

\newtheorem{defn}{Definition}

\def\Z{{\mathbb Z}}
\def\N{{\mathbb N}}

\def\G{\overline{G}}
\def\H{\overline{H}}

\def\H{\overline{H}}
\def\K{\overline{K}}
\def\G{\overline{G}}
\def\H{\overline{H}}

\begin{document}
\begin{center}
\LARGE {\bf \textsc{Geochromatic Number when Crossings are Independent}} 
\end{center}\bigskip

\begin{center}
\textsc{Debra L. Boutin} 

\textsc{Department of Mathematics and Statistics}

\textsc{Hamilton College, Clinton, NY 13323}

\textsc {\sl dboutin@hamilton.edu}

\end{center}

\begin{center}

\textsc{Alice Dean}

\textsc{Department of Mathematics and Statistics}

\textsc{Skidmore College, Clinton, NY 12866}

\textsc {\sl adean@skidmore.edu}

\end{center}

\title{Geochromatic Numbers}

\begin{abstract} A {\it geometric graph}, $\G$, is a graph drawn in the plane, with straight line edges and vertices in general position.  A {\it geometric homomorphism} between two geometric graphs $\G$, $\H$ is a vertex map $f:\G\to \H$ that preserves vertex adjacency and edge crossings. The {\it geochromatic number} of $\G$, denoted $X(\G)$, is the smallest integer $n$ so that there is a geometric homomorphism from $\G$ to some geometric realization of $K_n$.  Recall that  the {\it chromatic number} of an abstract graph $G$, denoted $\chi(G)$, is the smallest integer $n$ for which there is a graph homomorphism from $G$ to $K_n$. It is immediately clear that $\chi(G)\leq X(\G)$. This paper establishes some upper bounds on $X(\G)$ in terms of $\chi(G)$.  For instance, if all crossings are at distance at least 1 from each other, then $X(\G)\leq 3\chi(G)$.  However, there are more precise results.  If all crossing are at distance at least 2, then $X(\G)\leq \chi(G)+2$.  If all crossings are at distance at least 1, and there is a graph homomorphism $f: G \to K_n$ that maps no pair of edges that cross in $\G$ to the same edge in $K_n$, then $X(\G)\leq 2n$.  Finally, if $\chi(G)\in \{2,3\}$ and all crossings are at distance at least 1, then $X(\G)\leq 2\chi(G)$. \end{abstract}

\section{Introduction}

A geometric graph $\overline{G}$ is a simple graph drawn in the plane, on vertices in general position, with straightline edges.  Abstractly what we care about in a geometric graph is which pairs of vertices are adjacent and which pairs of edges cross.  In particular, two geometric graphs are said to be isomorphic if there is a bijection between their vertex sets that preserves adjacencies, non-adjacencies, crossings, and non-crossings.\medskip

A natural way to extend the ideas of abstract graph homomorphisms to the context of geometric graphs is to define a geometric homomorphism as a vertex map $f:\overline{G} \to \overline{H}$ that preserves both adjacencies and crossings (but not necessarily non-adjacencies or non-crossings).  If such a map exists, we write $\overline{G}\to\overline{H}$ and say \lq$\overline{G}$ is (geometrically) homomorphic to $\overline{H}$.'  There are many similarities between abstract graph homomorphisms and geometric graph homomorphisms, but there are also great contrasts.  Further, results that are straightforward in abstract graph homomorphism theory can become complex in geometric graph homomorphism theory.  \medskip

In abstract graph homomorphism theory (of simple graphs) two vertices cannot be identified under any homomorphism if and only if they are adjacent.  In Section \ref{subsec:generalobs}  we review some of the reasons why two vertices might not be able to be identified under any geometric homomorphism:  if they are adjacent;  if they are involved in a common edge crossing;  if they are endpoints of an odd length path each edge of which is crossed by a common edge; if they are endpoints of a path of length 2 whose edges cross all edges of an odd length cycle.\medskip

Recall that one definition for the {\it chromatic number} of a graph $G$, denoted $\chi(G)$, is the smallest integer $n$ so that $G\to K_n$.  By the transitivity of homomorphisms, if $G \to H$, then $\chi(G) \leq \chi(H)$.  Analogously, the {\it geochromatic number} of a geometric graph $\overline{G}$, denoted $X(\overline{G})$, is the smallest integer $n$ so that  $\overline{G}\to \overline{K}_n$, for some geometric $n$-clique $\overline{K}_n$. We immediately get that  $\overline{G} \to \overline{H}$ implies $X(\overline{G}) \leq X(\overline{H})$.  There are also other parameters whose relationships are preserved by geometric homomorphisms.  The {\it thickness} of a geometric graph, denoted $\theta(\overline{G})$, is the minimum number of plane layers of $\overline{G}$.  In \cite{BC2012} we see that that $\overline{G}\to\overline{H}$ implies both $\chi(G) \leq \chi(H)$ and $\theta(\overline{G})\leq \theta(\overline{H})$.\medskip

Two crossings in a geometric graph are said to be {\em independent} if they have no vertices in common.  Further, we  define the {\it distance} between two crossings as the minimum distance between the vertices of the first crossing and the vertices of the second. Thus, crossings are independent if they are at distance at least 1. In Section \ref{sec:geochrom} we focus on the geochromatic number of geometric graphs with given chromatic number and independent crossings.  In Section \ref{sec:pseudo} we briefly discuss a geometric graph coloring that is weaker than the geochromatic coloring, but might be mistaken for it. \medskip 

\section{Background}\label{subsec:generalobs}

For additional background on graph homomorphisms, see \cite{HN2004}.  For more information on geometric graphs, see \cite{P2004a, P2004b}. For other articles on geometric homomorphisms, see \cite{BC2012, BCDM2012, CoYo2013, Co2017, Co2018}.\medskip

In abstract graph homomorphism theory, every graph on $n$ vertices is homomorphic to $K_n$.  This is not true for geometric graphs.  In fact, two different geometric realizations of the same abstract graph are not necessarily homomorphic to each other.  For example, consider the two geometric realizations of $K_6$ given in Figure \ref{6cliques}.  The first has a vertex with all incident edges crossed, the second does not;  this can be used to prove that there is no geometric homomorphism from the first to the second.  The second has more crossings than the first; this can be used to prove that there is no geometric homomorphism from the second to the first.\medskip

\begin{figure}[hbt]
 \centering
 \scalebox{0.4}
 {
 \begin{tikzpicture}[scale=4]
	
	 \tikzstyle{vertex}=[draw, circle, thick] 
	 \tikzstyle{edge} = [draw,line width=1.5pt,-]
	 \node[vertex,   fill =black] (a) at (0,-.15){};
	 \node[vertex, fill = black] (b) at (0,.95){};
	 \node[vertex, fill = black] (c) at (1, .15){};
	 \node[vertex, fill = black] (d) at (-1, .15){};
	 \node[vertex, fill = black] (e) at (-.6,-1){};
	 \node[vertex, fill = black] (f) at (.6,-1){};
	 
	 \draw[edge] (a) -- (b); 
	 \draw[edge] (a) -- (c); 
	  \draw[edge] (a) -- (d);
	   \draw[edge] (a) -- (e);
	    \draw[edge] (a) -- (f);
	     \draw[edge] (b) -- (c);
	      \draw[edge] (b) -- (d);
	       \draw[edge] (b) -- (e);
	        \draw[edge] (b) -- (f);
	         \draw[edge] (c) -- (d);
	          \draw[edge] (c) -- (e);
	           \draw[edge] (c) -- (f);
	            \draw[edge] (d) -- (e);
	             \draw[edge] (d) -- (f);
	              \draw[edge] (e) -- (f);

	 \node[vertex,   fill =black] (A) at (3.05,1){};
	 \node[vertex, fill = black] (B) at (4,.5){};
	 \node[vertex, fill = black] (C) at (4, -.5){};
	 \node[vertex, fill = black] (D) at (3.05,-1){};
	 \node[vertex, fill = black] (E) at (2,-.5){};
	 \node[vertex, fill = black] (F) at (2,.5){};
	 
	 \draw[edge] (A) -- (B); 
	 \draw[edge] (A) -- (C); 
	  \draw[edge] (A) -- (D);
	   \draw[edge] (A) -- (E);
	    \draw[edge] (A) -- (F);
	     \draw[edge] (B) -- (C);
	      \draw[edge] (B) -- (D);
	       \draw[edge] (B) -- (E);
	        \draw[edge] (B) -- (F);
	         \draw[edge] (C) -- (D);
	          \draw[edge] (C) -- (E);
	           \draw[edge] (C) -- (F);
	            \draw[edge] (D) -- (E);
	             \draw[edge] (D) -- (F);
	              \draw[edge] (E) -- (F);
	           	  \end{tikzpicture}}
 \caption{Homomorphically distinct realizations of $K_6$.}\label{6cliques}
 \end{figure}
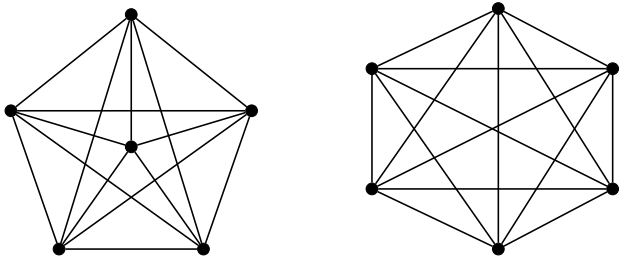

The following observations were proved in  \cite{BC2012}.  They are replicated here so that we can use them, but also to give a sense of some of the difficulties of finding geometric homomorphisms.\medskip 

\newpage

\begin{theorem}\label{thm:obs} \rm \cite{BC2012}\hskip-.5in\begin{enumerate}
\item[(a)]\label{obs1}  Adjacent vertices cannot be identified by any geometric homomorphism.
\item[(b)] \label{obs2} Endpoints of edges that cross cannot be identified by any geometric homomorphism. 
\item[(c)] \label{obs3} The endpoints of an odd-length path cannot be identified by any geometric homomorphism if there is an edge that crosses all the edges of the path. See the geometric graph on the left in Figure \ref{fig:(c),(d)}.
\item[(d)]\label{obs4} The endpoints of a path of length 2 cannot be identified by any geometric homomorphism if its edges cross all edges of an odd-length cycle. See the geometric graph on the right in Figure \ref{fig:(c),(d)}.\end{enumerate}\end{theorem}

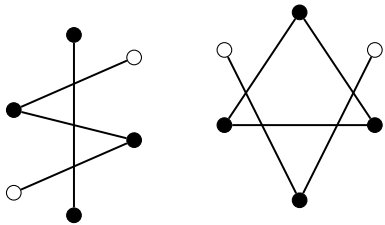
\begin{figure}[hbt]
 \centering
 \scalebox{0.5}
 {
 \begin{tikzpicture}[scale=4]
	
	 \tikzstyle{vertex}=[draw, circle, thick] 
	 \tikzstyle{edge} = [draw,line width=1.5pt,-]
	 \node[vertex,   fill =black] (a) at (0,0){};
	 \node[vertex, fill = black] (b) at (1,0){};
	 \node[vertex, fill = black] (c) at (.5,.75){};
	 \node[vertex, fill = white] (x) at (0,.5){};
	 \node[vertex, fill = white] (y) at (1,.5){};
	 \node[vertex, fill = black] (z) at (.5,-.5){};
	 \draw[edge] (a) -- (b) -- (c) -- (a);
	 \draw[edge] (x) -- (z)--(y);

	 \node[vertex, fill =black] (f) at (-1,.6){};
	 \node[vertex, fill = black] (e) at (-1,-.6){};
	 \node[vertex, fill = white] (1) at (-1.4,-.45){};
	 \node[vertex, fill = black] (2) at (-.6,-.1){};
	 \node[vertex, fill = black] (3) at (-1.4,.1){};
	 \node[vertex, fill = white] (4) at (-.6,.45){};
	 \draw[edge] (1) -- (2) -- (3) -- (4);
	 \draw[edge] (e) -- (f);
 \end{tikzpicture}}
 \caption{In each, the white vertices cannot be identified}
 \label{fig:(c),(d)}
 \end{figure}

Recall that the definition of geochromatic number of a geometric graph is a generalization of the chromatic number of a graph $G$, $\chi(G)$, as the smallest $n$ so that there is a graph homomorphism $G\to K_n$.  In particular we have the following definition. 

\begin{defn} \rm 
Let $\overline{G}$ be a geometric graph. The {\it geochromatic number} of  $\overline{G}$, denoted by $X(\overline{G})$, is the smallest positive integer $n$ such that $\overline{G} \to \overline{K}_n$ for some geometric clique $\overline{K}_n$.\end{defn}

The geochromatic number of $\overline{G}$ must be large enough not only to accommodate the adjacency relationships among vertices of $\overline{G}$, but also the crossing relationships among its edges.\medskip
 
In investigating geochromatic numbers, it makes sense to start with the smallest ones.  We can see that $X(\overline{G})=1$ if and only if $\overline{G}$  has no edges.\medskip

Since a geometric graph with a pair of crossing edges requires at least four colors in a geochromatic coloring, we immediately get  that $X(\overline{G}) = 2$ if and only if $\overline{G}$ is a bipartite plane geometric graph, and that $X(\overline{G}) = 3$ if and only if $\overline{G}$ is a 3-chromatic plane geometric graph. Since deciding whether $\chi(G) \leq 3$ is an NP-complete problem \cite{S1973}, deciding whether $X(\overline{G}) \leq 3$ is NP-complete also. Theorems in \cite{BC2012}  give conditions that are necessary but not sufficient for $X(\G)=4$, as well as  conditions that are sufficient but not necessary. In \cite{BC2012} we also see that for any $n$ there exists a thickness-2 geometric bipartite graph with geochromatic number $n$.  Thus even bounding both the chromatic number and  thickness of a geometric graph does not bound its geochromatic number.\medskip

Similarly, the number of crossings in a geometric graph often tells us little about its geochomatic number.  Above we mentioned that if $\G$ contains no crossing edges then $X(\G)=\chi(G)\leq 4$.  The next theorem tells us that even if $\G$ has large number of edge crossings, all we can say is that $X(\G)\geq 4$.\medskip

  \begin{theorem} \rm For every $k\in \N$ there is a geometric graph $\G$ with $k$ crossings and $X(\G)=4$. \end{theorem}

\begin{proof} Let $\K_4$ be the convex $4$-clique with vertices labeled 1, 2, 3, 4 around the convex hull. If $k=1$, then $\K_4$ is the requested geometric graph.  For $k >1$, start with the star graph $\overline{K}_{1, k}$, with its $k$ edges spanning an angle of $\frac{\pi}{2}$. Label the central vertex 0, and the leaves $1, 2, \ldots, k$. Cross those edges  with a single edge $\{k+1,k+2\}$. Call the resulting graph $\G$. Thus the crossings in $\G$ are the precisely the pairs $\{0,i\}$, $\{k+1,k+2\}$ for $i\in \{1,\ldots, k\}$.  Define $\beta: \G \to \K_4$ by $\beta(0)=1$, $\beta(i) = 3$ for $i\in \{1,\ldots, k\}$, $\beta(k+1)=2$ and $\beta(k+2) = 4$. Thus the edges of the star are mapped to $\{1,3\}$, while the edges of the crossing $K_2$ is mapped to $\{2,4\}$. Thus $\beta:\G \to \K_4$ is a geometric graph homomorphism.\end{proof}

It is essential that we be able to quickly determine, without necessarily visualizing the geometric graph, whether two edges in the graph cross.  The two following lemmas give us explicit tools.  Lemma \ref{lem:convex} will be used numerous times, sometimes implicitly, throughout this paper.\medskip 

\begin{lem}\label{lem:nonconvex} \rm If $S=\{u,v,x,y\}$ is a non-convex set of vertices in the plane, then no pair of edges with both endpoints in $S$ cross.\end{lem}

The above is clear since the geometric $\K_4$ built on the vertices in $S$ is a non-convex $\K_4$ and therefore has no crossing edges.\medskip 

\begin{lem}\label{lem:convex} \rm Let $n\geq 4$ vertices in convex position in the plane be labeled $1, 2, \ldots, n$ around their convex hull.  Let $\K_n$ be the convex $n$-clique built on these vertices.  Let $e_1=\{a_1, a_2\}$ and $e_2=\{b_1, b_2\}$ be disjoint edges in $\K_n$.  Without loss of generality, we may assume that $a_1<a_2$, $b_1<b_2$ and $a_1<b_1$.  Then $e_1$ and $e_2$ cross if and only if $a_1<b_1<a_2<b_2$.\end{lem}

The above is clear since a subset of a convex set is itself convex, and  edges cross if and only if their endpoints alternate as we move around their convex hull.\medskip

\section{Independent Crossings}\label{sec:geochrom}

\begin{defn} \rm In a geometric graph $\G$ we call an edge involved in a crossing a {\it crossing edge} and call a vertex incident to a crossing edge a {\it crossing vertex}. \end{defn}

\begin{defn} \rm 
Two crossings in $\overline{G}$ are said to be at {\em distance $d$} if the minimum distance between any vertex in one crossing and any vertex in the other crossing is $d$.  Two crossings are said to be {\em independent} if they have no vertices in common. That is, a two crossings are independent if they are at distance at least 1. \end{defn}

Notice that if all crossings in a geometric graph are at distance at least 2, then each crossing vertex is adjacent to exactly one other crossing vertex -- the other endpoint of the  crossing edge to which it is incident.  If all crossings are at distance at least 1, then a crossing vertex may be adjacent to many other crossing vertices -- only one of which is the other endpoint of the single crossing edge to which it is incident.  Other crossing vertices to which it may be adjacent are crossing vertices from other crossings that are at distance 1 from it.  Figure \ref{Fig:distance} shows examples of these situations.\medskip

 \begin{figure}[hbt]
 \centering
 \scalebox{0.3}
 {
 \begin{tikzpicture}[scale=4]
	 \tikzstyle{vertex}=[circle,minimum size=20pt,inner sep=0pt, fill=black!]
	 \tikzstyle{edge} = [draw,line width=2pt,-]
	 \node[vertex] (v00) at (0,0){};
	 \node[vertex] (v10) at (1,0){};
	 \node[vertex] (v01) at (0,1){};
	 \node[vertex] (v11) at (1,1){};
	 \node[vertex] (v21) at (2,1){};
	 \node[vertex] (v31) at (3,1){};
	 \node[vertex] (v41) at (4,1){};
	 \node[vertex] (v30) at (3,0){};
	 \node[vertex] (v40) at (4,0){};
	 \draw[edge] (v00) -- (v11) -- (v21) -- (v31) -- (v40);
	 \draw[edge] (v01) -- (v10);
	 \draw[edge] (v30) -- (v41);
	 \node[vertex] (v60) at (6,0){};
	 \node[vertex] (v70) at (7,0){};
	 \node[vertex] (v80) at (8,0){};
	 \node[vertex] (v90) at (9,0){};
	 \node[vertex] (v61) at (6,1){};
	 \node[vertex] (v71) at (7,1){};
	 \node[vertex] (v81) at (8,1){};
	 \node[vertex] (v91) at (9,1){};
	 \draw[edge] (v60) -- (v71) -- (v81) -- (v90);
	 \draw[edge] (v61) -- (v70);
	 \draw[edge] (v91) -- (v80);
 \end{tikzpicture}}
 \caption{At left: crossings at distance 2; at right: crossings at distance 1}
 \label{Fig:distance}
 \end{figure}
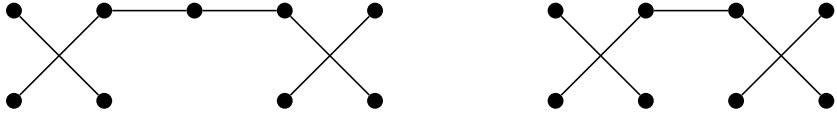

\begin{theorem}\label{thm:dist2}\rm
If all pairs of crossings in $\overline{G}$ are at distance at least 2, then $X(\overline{G}) \le \chi(G) + 2$.
\end{theorem}

\begin{proof}  Suppose $\chi(G) = n$. Let $\K_{n+2}$ be a convex geometric $(n+2)$-clique with vertices labeled $1,\ldots, n+2$.  Associate the complete subgraph induced by the  set $\{1, 2, \ldots, n\}$ with  $K_n$.  By assumption on $\chi(G)$, there is a graph homomorphism $\alpha:G \to K_n$. \medskip

We will modify $\alpha:G\to K_n$ to a homomorphism $\beta:G\to K_{n+2}$ so that $\beta$ takes crossing of $\G$ to crossings in $\K_{n+2}$. We will only modify $\alpha$ on crossing vertices.; the images of non-crossing vertices will be the same under $\beta$ and $\alpha$. The modified homomorphism $\beta$ will map some crossing vertices to one of $n+1,n+2$.  The remaining crossing vertices in that crossing will maintain their image under $\alpha$. Since $K_{n+2}$ has edges between all pairs of vertices, as long as the modification of $\alpha$ to $\beta$ maintains the property that no edge has endpoints mapped to the same vertex, $\beta$ will be a graph homomorphism. With wise choices justified by Lemma \ref{lem:convex}, we can ensure that the resulting graph homomorphism $\beta:\G \to \K_{n+2}$ is a geometric graph homomorphism.\medskip

Since pairs of crossing edges are at distance at least 2, they are independent. Thus no two crossing share a vertex.  Therefore, if we modify $\alpha$ on crossing vertices in one crossing, it does not affect any modification we may have previously made on another crossing.  Further, again since pairs of crossing edges are at distance at least 2, each crossing vertex in $\G$ is adjacent to exactly one other crossing vertex - the other endpoint of its crossing edge. Thus crossing vertices from distinct crossings are nonadjacent.  Thus in defining $\beta$, we run no risk of mapping two crossing vertices from different crossings to the same vertex in $\{n+1, n+2\}$. These facts allow us to modify our homomorphism $\alpha:G\to K_n$ to a new homomorphism $\beta:\G\to \K_{n+2}$ one crossing at a time.\medskip 

Suppose $e_1=\{u,v\}$ and $e_2=\{x,y\}$ are a pair of crossing edges in $\G$. Note that since $\alpha$ is a graph homomorphism, and $K_n$ is a simple graph, $\alpha(x)\ne \alpha(y)$ and $\alpha(u) \ne \alpha(v)$. Under the homomorphism $\alpha$, edge $\alpha(e_1)$ may have 0, 1 or 2 vertices in common with edge $\alpha(e_2)$.  Thus the vertices $\alpha(u), \alpha(v), \alpha(x), \alpha(y)$ either induce  Case 1): two non-incident edges, or Case 2): two incident edges, or Case 3): a single edge in $K_n$.\medskip 

\noindent{\bf Case 1)} Suppose the edges are non-incident.  If $\alpha(e_1), \alpha(e_2)$ cross in $\K_n$, then we need not modify $\alpha$ on the vertices of this crossing pair.   If the disjoint edges $\alpha(e_1), \alpha(e_2)$ don't cross in $\K_n$ then, without loss of generality, the two possible relationships between these edges and the additional vertices, $n+1$ and $n+2$, are illustrated in Figure \ref{Cases1ab}.  That is, either Case 1a): the additional vertices fall between vertices of the two edges, without loss of generality say between $\alpha(y)$ and $\alpha(u)$, or Case 1b): the additional vertices fall between the endpoints of one of the edges, without loss of generality say between $\alpha(x)$ and $\alpha(y)$.\medskip 

{\bf Case 1a)} Referring to Figure \ref{Cases1ab}, $1\leq \alpha(u) < \alpha(v) < \alpha(x) < \alpha(y) \leq n$. Modify the mapping so that $\beta(v)= n+1$ and $\beta(x)=n+2$, while $\beta=\alpha$ on $u$ and $y$.  Then $\beta(u)<\beta(y)<\beta(v) <\beta(x)$ and therefore $\beta(e_1)$ and $\beta(e_2)$ cross in $\K_{n+2}$.\medskip 

{\bf Case 1b)} Referring to Figure \ref{Cases1ab}, $1\leq \alpha(y) < \alpha(u) < \alpha(v) < \alpha(x)\leq n$. Modify the mapping so that $\beta(v)= n+1$, while $\beta=\alpha$ on $u,x$ and $y$.  Then $\beta(y)< \beta(u)<\beta(x)<\beta(v)$ and therefore $\beta(e_1)$ and $\beta(e_2)$ cross in $\K_{n+2}$.\medskip

 \begin{figure}[hbt]
 \centering
 \scalebox{0.6}
 {\begin{tikzpicture}[scale=4]
	 \tikzstyle{vertex}=[draw, circle, thick, fill = black] 
	 \tikzstyle{edge} = [draw,line width=1.25pt,-]
	 \node[above] at (1.25,2.3){Case 1a)};
	 \node[above] at (1.25,2.15){Add'l Vertices};
	 \node[above] at (1,1.95){\phantom{aasdfassdf} \Large  $\cdots$};
	 \node[left] at (-.5,.85){Case 1b)};	
	 \node[left] at (-.35,.70){Add'l Vertices};	
	 \node[left] at (.055,.775){ \phantom{assdasdf} \Large $\vdots$};
	 \node[right,color=white] at (3.1,.75){this};
	 \node[vertex] (x) at (.5,0){};
	 \node[left] at (.5,0){\small $\alpha(x) \ $};
	 \node[vertex] (x) at (.5,0){};
	 \node[left] at (.5,0){\small $\alpha(x) \ $};
	 \node[vertex] (v) at (2,0){};
	 \node[left] at (2,0){\small $\alpha(v) \ $};
	 \node[vertex] (y) at (.5,1.5){};
	 \node[left] at (.5,1.5){\small $\alpha(y) \ $};
	 \node[vertex] (u) at (2,1.5){};
	 \node[left] at (2,1.5){\small $\alpha(u) \ $};
	 \node[vertex] (1up) at (1,2){};
	 \node[vertex] (2up) at (1.5,2){};
	 \node[vertex] (1left) at (0,.5){};
	 \node[vertex] (2left) at (0,1){};
	 \draw[edge] (u) -- (v);
	 \draw[edge] (x) -- (y);
 \end{tikzpicture}}
 \caption{Cases 1a) \& 1b)}
 \label{Cases1ab}
 \end{figure}
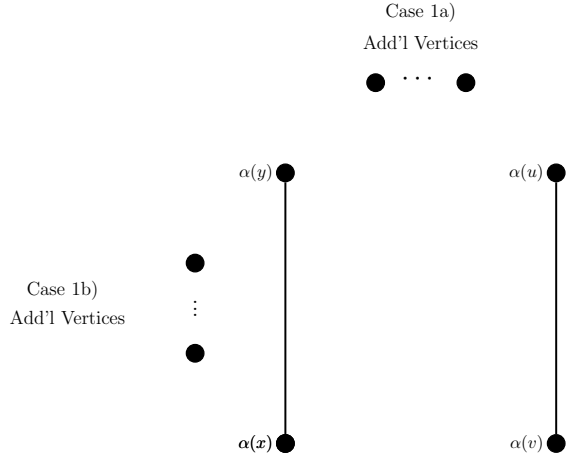
 
\noindent{\bf Case 2)} Suppose $\alpha(e_1)$ and $\alpha(e_2)$ are distinct but incident edges in $K_n$.  In this case, without loss of generality, assume that $\alpha(v)=\alpha(x)$. The two possible relationships between the associated crossing vertices and the additional vertices, $n+1$ and $n+2$, are illustrated in Figure \ref{Cases2ab}. That is, either Case 2a): the additional vertices fall between leaves of the path, say between $\alpha(y)$ and $\alpha(u)$, or Case 2b): they fall between the endpoints of one of the edges, say between $\alpha(x)$ and $\alpha(y)$.\medskip   
 
{\bf Case 2a)}  Referring to Figure \ref{Cases2ab}, $1\leq \alpha(u) < (\alpha(v) =\alpha(x)) < \alpha(y) \leq n$. Modify the mapping so that $\beta(v)= n+1$ and $\beta(y)= n+2$, while $\beta=\alpha$ on $u$ and $x$. Then $\beta(u)< \beta(x)<\beta(v)<\beta(y)$ and therefore $\beta(e_1)$ and $\beta(e_2)$ cross in $\K_{n+2}$.\medskip

{\bf Case 2b)} Referring to Figure \ref{Cases2ab}, $1\leq \alpha(y) < \alpha(u) < (\alpha(v) =\alpha(x))  \leq n$.  Modify the mapping so that $\beta(v) = n+1$, while $\beta=\alpha$ on $u,x$ and $y$. Then $\beta(y) < \beta(u) < \beta(x) < \beta(v)$ and therefore $\beta(e_1)$ and $\beta(e_2)$ cross in $\K_{n+2}$\medskip

\begin{figure}[hbt]
 \centering
 \scalebox{0.6}
 {
 \begin{tikzpicture}[scale=4]
	 \tikzstyle{vertex}=[draw, circle, thick, fill = black] 
	 \tikzstyle{edge} = [draw,line width=1.25pt,-]
	 
	 \node[right,color=white] at (-.8,.75){this};
	  \node[above] at (-2.75,2.32){Case 2a)};
	 \node[above] at (-2.75,2.2){Add'l Vertices};
	 \node[above] at (-2.959,1.93){\phantom{aasdfassdf} \Large  $\cdots$};
	 \node[left] at (-4.25,.85){Case 2b)};	
	 \node[left] at (-4.25,.70){Add'l Vertices};	
	 \node[left] at (-3.95,.79){\phantom{aasdfassdf} \Large  $\vdots$};
	 \node[vertex] (xv) at (-2.68,0){};
	 \node[below] at (-2.5,-.1){\small $\alpha(v)=\alpha(x)$ \phantom{adasdcsf}};
	 \node[vertex] (y) at (-3.5,1.5){};
	 \node[left] at (-3.5,1.5){\small $\alpha(y) \ $};
	 \node[vertex] (u) at (-2,1.5){};
	 \node[left] at (-2,1.5){\small $\alpha(u) \ $};
	 \node[vertex] (1up) at (-3,2){};
	 \node[vertex] (2up) at (-2.5,2){};
	 \node[vertex] (1left) at (-4,.5){};
	  \node[vertex] (2left) at (-4,1){};
	 \draw[edge] (u) -- (xv);
	 \draw[edge] (xv) -- (y);
 \end{tikzpicture}}
 \caption{Cases 2a) \& 2b)}
 \label{Cases2ab}
 \end{figure}

\noindent {\bf Case 3)} Suppose $\alpha(e_1)=\alpha(e_2)$. In this case, without loss of generality, assume $\alpha(u)=\alpha(y)$ and $\alpha(v)=\alpha(x)$, and $\alpha(u) < \alpha(v)$. The relationship of their crossing vertices with the additional vertices, $n+1$ and $n+2$, is illustrated in Figure \ref{Case5}. \medskip 

  \begin{figure}[hbt]
 \centering
 \scalebox{0.6}
 {
 \begin{tikzpicture}[scale=4]
	 \tikzstyle{vertex}=[draw, circle, thick, fill=black] 
	 \tikzstyle{edge} = [draw,line width=1.25pt,-]
	 
	 \node[right,color=white] at (-2,.75){this};
	 
	 \node[left] at (-4.25,.85){Case 3)};	
	 \node[left] at (-4.25,.70){Add'l Vertices};	
	 \node[left] at (-3.95,.8){\phantom{aasdfassdf} \Large  $\vdots$};
	 \node[vertex] (xv) at (-3.32,0){};
	 \node[below] at (-3.25,-.1){\small $\alpha(v)=\alpha(x)$ \phantom{ada}};
	 \node[vertex] (uy) at (-3.32,1.5){};
	 \node[above] at (-3.25,1.6){\small $\alpha(u)=\alpha(y)$ \phantom{ada}};
	 \node[vertex] (1left) at (-4,.5){};
	 \node[vertex] (2left) at (-4,1){};
	 \draw[edge] (uy) -- (xv);
 \end{tikzpicture}}
 \caption{Case 3)}
 \label{Case5}
 \end{figure}

 Modify the mapping so that  $\beta(y)=n+2$, and $\beta(v)={n+1}$, while $\beta=\alpha$ on $u$ and $x$. Then $\beta(u)<\beta(x)<\beta(v)<\beta(y)$, and thus $\beta(e_1)$ crosses $\beta(e_2)$.\medskip  

Thus after modifying $\alpha$ to $\beta$ on each crossing, we've defined $\beta:\G\to \K_{n+2}$ as a geometric graph homomorphism.  Thus $X(\G)\leq \chi(G)+2$.\end{proof}

There are easy examples of geometric graphs where this bound is not best possible.  A plane geometric graph, and a  convex $\K_4$ fall into this category.  It would be interesting to find criteria for when this result is best possible.\medskip

Now we move to considering geometric graphs in which all crossings are independent, but not necessarily at distance at least 2. The technique used in the proofs of Theorems \ref{thm:chinX2n} and \ref{thm:chinX3n} is similar to the one used in the proof of Theorem \ref{thm:dist2}.  The difference is that in Theorem \ref{thm:dist2} since the minimum distance between two pairs of crossing edges was 2, no pair of crossing vertices from distinct edge crossings could be adjacent.  Thus there was no risk that in modifying $\alpha$ to achieve $\beta$ on distinct crossing pairs we might map adjacent crossing vertices to the same vertex in our geometric clique.  When crossings are independent but not of distance at least 2, we have to be a bit more careful.  In this situation, crossing vertices from distinct crossings may indeed be adjacent.  One solution to this is to base the modification of $\alpha(v)$ on the image $\alpha(v)$ itself. In this technique, for all $v\in V(G)$ we will define $\beta(v)$ to be in $\{\alpha(v), \alpha(v)+n, \alpha(v)+2n\}$.  Since $\alpha:G\to K_n$ is a homomorphism,  if $u,v$ are adjacent vertices in $G$, then $\alpha(v)\ne \alpha(u)$, and each is between $1$ and $n$. Then the set of the potential images $\beta(u)$ and $\beta(v)$ are $\{\alpha(u), \alpha(u)+n, \alpha(u)+2n\}$ and $\{\alpha(v), \alpha(v)+n, \alpha(v)+2n\}$.  Since these sets are disjoint,  $\beta(v)\ne \beta(u)$ as desired. This is the key behind the following proof. 


\begin{theorem}\label{thm:chinX2n} \rm Let $\G$ be a geometric graph with independent crossings.  Let $n$ be an integer so that there exists a homomorphism $\alpha: G \to K_n$ in which no pair of crossing edges are mapped to the same edge. Then $X(G) \leq 2n$.
\end{theorem}

\begin{proof}  Let $\K_{2n}$ be a convex geometric $2n$-clique with vertices labeled $1,\ldots, 2n$ around the convex hull.  Associate the complete subgraph induced on the convex set $\{1, 2, \ldots, n\}$ with  $K_n$.  By assumption, there is a graph homomorphism $\alpha:G \to K_n$ in which no pair of  edges that cross in $\G$ are mapped to the same edge in $K_n$.  We will modify this to a geometric graph homomorphism $\beta: \G \to \K_{2n}$.\medskip 

For this proof we will define $\beta:G\to K_{2n}$ by adding $n$ to images under $\alpha$ of some crossing vertices. As argued earlier,  for any $u,v \in V(G)$ for which $\alpha(u) \ne \alpha(v)$, with $\beta(u)\in\{\alpha(u), \alpha(u)+n\}$ and $\beta(v)\in\{\alpha(v), \alpha(v)+n\}$, we have $\beta(u)\ne \beta(v)$. Thus since $\alpha$ is a graph homomorphism and $\beta$ maintains inequality of vertex images, $\beta$ is also a graph homomorphism.\medskip

Let $e_1=\{u,v\}, e_2=\{x,y\}$ be crossing edges in $\G$.  By our assumption on $\alpha$, it sends no pair of crossing edges in $\G$ to the same edge of $\K_n$.  Then $\alpha(e_1)$ and $\alpha(e_2)$ share at most one vertex. Thus we have two cases, Case 1): $\alpha(e_1)$ and $\alpha(e_2)$ are disjoint,  or Case 2): $\alpha(e_1)$ and $\alpha(e_2)$ share precisely one vertex.  These are identical to Case 1) and Case 2) of the proof of Theorem \ref{thm:dist2}.  Further, the relationships between the additional vertices $n+1, \ldots, 2n$ and the vertices of $\alpha(e_1)$ and $\alpha(e_2)$ fall into exactly the same subcases here as they did in the proof of Theorem \ref{thm:dist2}.\medskip

{\bf Case 1a)} Referring to Figure \ref{Cases1ab} we see that $1\leq \alpha(u) < \alpha(v) < \alpha(x) < \alpha(y)\leq n$. Define $\beta$ so that $\beta(v)=\alpha(v)+n$ and $\beta(x)=\alpha(x)+n$, while $\beta=\alpha$ on $u$ and $y$.  Then $\beta(u)<\beta(y)<\beta(v) <\beta(x)$ and therefore $\beta(e_1)$ and $\beta(e_2)$ cross.\medskip

{\bf Case 1b)} Referring to Figure \ref{Cases1ab} we see that $1\leq \alpha(y) < \alpha(u) < \alpha(v) < \alpha(x)\leq n$. Define $\beta$ so that $\beta(v)=\alpha(v)+n$, while $\beta=\alpha$ on $u, x$ and $y$.  Then $\beta(y)< \beta(u)<\beta(x)<\beta(v)$ and therefore $\beta(e_1)$ and $\beta(e_2)$ cross.\medskip  

{\bf Case 2a)} Referring to Figure \ref{Cases2ab}, we see that $1\leq \alpha(u)< (\alpha(v)=\alpha(x))<\alpha(y)\leq n$. Define $\beta(v)=\alpha(v)+n$ and $\beta(y)=\alpha(y)+n$ while $\beta=\alpha$ on $u$ and $x$.  Then $\beta(u) < \beta(x) < \beta(v) < \beta(y)$ and therefore $\beta(e_1)$ and $\beta(e_2)$ cross.\medskip

{\bf Case 2b)}  Referring to Figure \ref{Cases2ab}, we see that  $1\leq \alpha(y) < \alpha(u)< (\alpha(v)=\alpha(x))\leq n$. Define $\beta(v)=\alpha(v)+n$, while $\beta=\alpha$ on $u, x$ and $y$.  Then $\beta(y)< \beta(u)<\beta(x)<\beta(v)$ and therefore $\beta(e_1)$ and $\beta(e_2)$ cross.\medskip

Thus $\beta:\G\to \K_{2n}$ is a geometric graph homomorphism.\end{proof}

Note that the above technique does not always work if a pair of crossing edges are mapped to a single edge in $K_n$ by the homomorphism $\alpha$.  For example, if a pair of crossing edges are mapped to $\{2,4\}$ in $\K_{10}$ and we added $n=5$ to a subset of $\alpha(u), \alpha(v), \alpha(x)$ and $\alpha(y)$, the only valid results would be either a single edge, a path of length 2, or two disjoint edges.  In particular, the result cannot be a crossing pair in $\K_{10}$. However, if $\chi(G) = 2$ or $3$ this will not be a problem.  This is stated as Theorem \ref{thm:2n3}. 
 But first, the following theorem gives slightly weaker conclusion in the more general case where we allow homomorphisms that map crossing edges in $\G$ to the same edge in $K_n$. 


\begin{theorem}\label{thm:chinX3n} \rm If $\G$ is a geometric graph with independent crossings, then $X(G) \leq 3\chi(G)$.
\end{theorem}

\begin{proof}  Let $\chi(G)=n$. Let $\K_{3n}$ be a convex geometric $3n$-clique with vertices labeled $1,\ldots, 3n$.  Associate the $n$-clique induced by the convex set $\{1, 2, \ldots, n\}$ with  $K_n$.  Since $\chi(G)=n$, there is a graph homomorphism $\alpha:G \to K_n$.  We will modify this to a geometric homomorphism $\beta: \G \to \K_{3n}$.\medskip 

We will redefine $\alpha$ to $\beta$ on certain crossing vertices of $\G$, by adding 0, or $n$ or $2n$ to $\alpha(v)$. With this definition, as argued previously, if $\alpha(u)\ne \alpha(v)$  then $\beta(v)\ne \beta(u)$.  This guarantees that $\beta$ preserves inequality given by $\alpha$, and is thus a graph homomorphism.  With careful choice, we can ensure $\beta:\G \to \K_{3n}$ is a geometric homomorphism.\medskip

Let $e_1=\{u,v\}, e_2=\{x,y\}$ be crossing edges in $\G$.  Note that $\alpha(e_1)$ and $\alpha(e_2)$ are either disjoint, share precisely one vertex, or share both vertices.\medskip

Associate the geometric $2n$-clique induced by the convex set $\{1, 2, \ldots, 2n\}$ with  $\K_{2n}$. If $\alpha(e_1)$ and $\alpha(e_2)$ share at most one vertex, use the method of Theorem \ref{thm:chinX2n} to modify the homomorphism $\alpha:G\to K_n$ to a geometric homomorphism $\beta:\G\to \K_{2n}$ where we identify the convex induced subgraph of $\K_{3m}$ by the vertices $\{1, 2, \ldots, 2n\}$ as $\K_{2n}$. By Theorem \ref{thm:chinX2n}, this modified homomorphism maps crossing edges of $\G$ whose images under $\alpha$ are not identical, to crossing edges in $\K_{2n}\subseteq \K_{3n}$.\medskip

Now consider the one case that is not covered in Theorem \ref{thm:chinX2n}.  Suppose that $\alpha(e_1)=\alpha(e_2)$. Without loss of generality,  suppose that $\alpha(u)=\alpha(y)$ and $\alpha(v)=\alpha(x)$ and $\alpha(u)<\alpha(v)$.  Define $\beta(x)=\alpha(x)+n, \beta(u) = \alpha(u) + 2n$, while $\beta=\alpha$ on $y$ and $v$.  Then $\beta(y)<\beta(v) < \beta (x) < \beta(u)$.  Thus $\beta(e_1)$ and $\beta(e_2)$ cross.\medskip 

Thus after modifying $\alpha$ to $\beta$ on each crossing, we have $\beta: \G \to \K_{3n}$ as a geometric graph homomorphism.\end{proof}

Our final theorem regards geometric graphs whose underlying abstract graph has chromatic number 2 or 3.  The technique used to prove the following theorem is slightly different than the one used earlier.  In particular we embed $\K_{\chi(G)}$ into $\K_{2\chi(G)}$ in a slightly different way and then add 1 to certain crossing vertices. This technique works well for situations when we can be assured that images of crossing edges are either incident or identical.  It does not always work when images of crossing edges are disjoint.

\begin{theorem}\label{thm:2n3} \rm If $\G$ has independent crossings and $\chi(G)\in\{2, 3\}$, then $X(\G) \leq 2\chi(G)$. \end{theorem} 

\begin{proof} 
 
 As argued in the previous theorems, because the crossings of $\G$ are independent, given a graph homomorphism $\alpha:G\to K_{\chi(G)}$, we may modify $\alpha$ to $\beta$ on  vertices of one crossing at a time, with the images under $\beta$ based on the images under $\alpha$. Our result will be $\beta:\G\to K_{2\chi(G)}$, a geometric graph homomorphism.  Below this is done explicitly in each of the cases $\chi(G)=3$ and $\chi(G) = 2$.\medskip

 If $\chi(G)=3$, begin with $\alpha:G\to K_3$ and modify to $\beta:\G\to \K_6$, the convex 6-clique with vertices labeled 1, 2, 3, 4, 5, 6 around the convex hull. Here we will consider the $3$-clique on vertices $1, 3, 5$ to be $K_3$.\medskip 
 
 Consider a crossing  in $\G$ with edges $e_1=\{u,v\}$ and $e_2=\{x,y\}$.  Either $\alpha(e_1)$ and $\alpha(e_2)$ are incident in $K_3$ or they are identical. If they are incident, without loss of generality, assume $\alpha(u)=1, \alpha(v)=\alpha(x) = 3$ and $\alpha(y)=5$.  Modify so that $\beta(x) =\alpha(x)+1 = 4$ while $\beta=\alpha$ on $u,v$ and $y$.  Then $\beta(e_1)= \{1,3\}$ and $\beta(e_2)= \{2,5\}$, a pair of crossing edges in $\K_6$.   If $\alpha(e_1)$ and  $\alpha(e_2)$ are identical, without loss of generality, we may assume $\alpha(u)= \alpha(y)=1$ and $\alpha(v)=\alpha(x)=3$. Modify so that $\beta(y)=\alpha(y)+1=2$, and $\beta(x) = \alpha(x) + 1= 4$ while $\beta=\alpha$ on $u$ and $v$.  Then $\beta(e_1)= \{1,3\}$ and $\beta(e_2) = \{2, 4\}$, crossing edges in $\K_6$.  After applying this process to each crossing in $\G$, we have achieved $\beta:\G \to \K_6$, a geometric graph homomorphism. \medskip

For the case $\chi(G)=2$, we begin with $\alpha:G\to K_2$ and modify to $\beta:\G\to \K_4$, the convex 4-clique with vertices labeled 1, 2, 3, 4 around the convex hull.  Here we will consider the $2$-clique on vertices $1,3 $ to be $K_2$. Since $\alpha$ necessarily maps crossing edges of $\G$ to the same edge of $K_2$, we can assume crossing edges $e_1=\{u, v\}$ and $e_2=\{x, y\}$ with $\alpha(u)=\alpha(y)=1$ and $\alpha(v)=\alpha(x)=3$. Define $\beta(y)=\alpha(y)+1=2 $ and $\beta(x)= \alpha(x) + 1=4$, while $\beta=\alpha$ on $u$ and $v$. Thus $\beta(e_1)= \{1, 3\}$ and $\beta(e_2) = \{2, 4\}$, crossing edges in $\K_4$.  After applying this process to each crossing in $\G$, we have achieved $\beta:\G \to \K_4$, a geometric graph homomorphism. \medskip. \end{proof}

\section{Pseudo-geochromatic number}\label{sec:pseudo}

It is natural to expect that the definition of $X(\overline{G})$ might be the same as  \lq the smallest integer $n$ so that the vertices of $\overline{G}$ can be colored with $n$ colors so that distinct colors are given to all vertex pairs corresponding to edges and to all vertex quadruples corresponding to crossing edges.'  Call this parameter the {\it pseudo-geochromatic number}, $X'(G)$. However, in general $X(\G)\ne X'(\G)$.  Consider $\overline{G}$ the geometric graph in Figure \ref{fig:diffcolor}. The coloring of the vertices as shown meets the conditions of a pseudo-geochromatic coloring.  A little arguing shows $X'(\G)=5$.  However, as shown in Theorem \ref{thm:obs}(d), the two vertices colored 5 cannot be identified under any geometric homomorphism.  Further, by Theorem \ref{thm:obs}(b), no other vertex pair of $\overline{G}$ can be identified by any homomorphism since each pair is involved in a common crossing. Thus, we conclude that $X(\overline{G})=6$.  This shows that the conditions for the pseudo-geochromatic number are weaker than the those of the geochromatic number.  

\begin{figure}[hbt]
 \centering
 \scalebox{0.5}
 {
 \begin{tikzpicture}[scale=4]
	
	 \tikzstyle{vertex}=[draw, circle, thick] 
	 \tikzstyle{edge} = [draw,line width=1.5pt,-]
	 \node[vertex,   fill = black] (a) at (-.5,0){};
	 \node[left] at (-.6,0) {{\Large \boldmath $1$}};
	 \node[vertex, fill = black] (b) at (.5,0){};
	 \node[right] at (.6,0) {{\Large \boldmath $2$}};
	 \node[vertex, fill = black] (c) at (0,.75){};
	 \node[above] at (0,.85) {{\Large \boldmath $3$}};
	 \node[vertex, fill = black] (x) at (-.5,.5){};
	 \node[above] at (-.55,.6) {{\Large \boldmath $5$}};
	 \node[vertex, fill = black] (y) at (.5,.5){};
	 \node[above] at (.535,.6) {{\Large \boldmath $5$}};
	 \node[vertex, fill = black] (z) at (0,-.4){};
	 \node[below] at (0,-.55) {{\Large \boldmath $4$}};
	 \draw[edge] (a) -- (b) -- (c) -- (a);
	 \draw[edge] (x) -- (z)--(y);
 \end{tikzpicture}}
 \caption{$X'(\G)<X(\G)$}
 \label{fig:diffcolor}
 \end{figure}
 
 Using this example as inspiration, we can create a family of geometric graphs with $X(\G)$ arbitrarily larger than $X'(\G)$.

\begin{figure}[hbt]
 \centering
 \scalebox{0.5}
 {
 \begin{tikzpicture}[scale=4]
	
	 \tikzstyle{vertex}=[draw, circle, thick] 
	 \tikzstyle{vertex2}=[ draw, circle] 
	 \tikzstyle{edge} = [draw,line width=1.5pt,-]
	 
	 \node[vertex, fill = black] (1) at (-1,1.732){};
	 \node[left] at (-1.15,1.832) {{\Large \boldmath $1$}};
	 
	 \node[vertex, fill = black] (2) at (-1.618,1.176){};
	 \node[left] at (-1.768,1.226) {{\Large \boldmath $2$}};
	 
	 \node[vertex, fill = black] (3) at (-1.956,.416){};
	 \node[left] at (-2.156,.416) {{\Large \boldmath $3$}};
	 
	 \node[vertex, fill = black] (4) at (-1.956,-.416){};
	 \node[left] at (-2.156,-.416) {{\Large \boldmath $4$}};

	 \node[vertex, fill = black] (11) at (2,0){};
	 \node[right] at (2.2,0) {{\Large \boldmath $2m+1$}};
	 \node[vertex, fill = black] (12) at (1.828, .814){};
	 \node[right] at (2.028, .814) {{\Large \boldmath $2m+2$}};
	 \node[vertex, fill = black] (13) at (1.338, 1.486){};
	 \node[right] at (1.538, 1.536) {{\Large \boldmath $2m+3$}};
	 \node[vertex, fill = black] (14) at (.618,1.902){};
	 \node[right] at  (.818,2.002) {{\Large \boldmath $2m+4$}};
	 	 
	 \node[vertex, fill = black] (6) at (-1,-1.732){};
	 \node[below] at (-1.2,-1.872) {{\Large \boldmath $m+1$}};
	 \node[vertex, fill = black] (7) at (-.21,-1.990){};
	 \node[below] at (-.28,-2.170) {{\Large \boldmath $m+2$}};
	 \node[vertex, fill = black] (8) at (.618,-1.902){};
	 \node[below] at (.618,-2.050) {{\Large \boldmath $m+3$}};
	 \node[vertex, fill = black] (9) at (1.338, -1.486){};
	 \node[below] at (1.338, -1.556) {{\Large \boldmath $m+4$}};

	 \node[vertex2, fill = gray] (10a) at (.07, 1.986){};
	 \node[vertex2, fill = gray] (10b) at (-.484, 1.940){};
	 \node[vertex2, fill = gray] (10) at (-.21,1.990){};
	 
	 \node[vertex2, fill = gray] (5a) at (-1.766, -.938){};
	 \node[vertex2, fill = gray] (5b) at (-1.438, -1.390){};
	 \node[vertex2, fill = gray] (5) at (-1.618,-1.176){};
	 
	 \node[vertex2, fill = gray] (15) at (1.828, -.814){};
	 \node[vertex2, fill = gray] (15a) at (1.696, -1.06){};
	 \node[vertex2, fill = gray] (15b) at (1.922, -.552){};

	 \draw[edge] (6) -- (1) -- (11) -- (6);
	 \draw[edge] (7) -- (2) -- (12) ;
	 \draw[edge] (8) -- (3) -- (13) ;
	 \draw[edge] (9) -- (4) -- (14);
 \end{tikzpicture}}
 \caption{}
 \label{fig:NewThm}
 \end{figure}

\begin{theorem} \rm For any $n\in \Z^+$ there is a geometric graph $\G$ for which $X(\G)-X'(\G) = n$. \end{theorem}

\begin{proof} Let $n\in \Z^+$, and $m=n+1$.  Construct the geometric graph on $3m$ vertices given in Figure \ref{fig:NewThm}.  There is the obvious geometric homomorphism to the convex $K_{3m}$.  Thus $X(\G)\leq 3m$.\medskip

Let $p:V(\G) \to \{1, 2, \ldots, 3m\}$ be a psuedo-geochromatic coloring, and $g:V(\G) \to \{1, 2, \ldots, 3m\}$ a geochromatic coloring. \medskip

Let us start with the induced subgraph on $v_1, v_{m+1}, v_{2m+1}$. Since these vertices  form a 3-cycle, each must have a different color under either coloring.  Let $p(v_1)=g(v_1)=1$, $p(v_{m+1})=g(v_{m+1})=m+1$, and  $p(v_{2m+1})=g(v_{2m+1})=2m+1$.  We will add, and color, the paths of length 2 one at time.  Add the path on $v_2, v_{m+2}$, and $v_{2m+2}$.  As drawn, the edge $\{v_2, v_{m+2}\}$ will cross both $\{1,m+2\}$ and $\{m+1, 2m+2\}$  Thus under both colorings, $v_2$ cannot be colored $1, m+1$, or $2m+1$. So let $p(v_2)=g(v_2)=2$.  Since the edges $\{ v_2, v_{m+2}\}$ and $\{v_2,v_{2m+2}\}$ do not cross, we may let  $p(v_{m+2})=p(v_{2m+2})=m+2$. However, by Theorem \ref{thm:obs}(d) since this path of length 2 crosses each edge of our $C_3$, the colors $g(v_{m+1}), g(v_{2m+1})$ must be distinct.  Thus let  $g(v_{m+2})=m+2$ and $g(v_{2m+2})=2m+2$.  We continue in this manner with each new path of length 2.  For $3\leq i \leq m-1$, consider $v_i, v_{m+i}, v_{2m+i}$.  Because its incident edges cross edges incident to all previous vertex colors $v_i$ will have to be colored differently than any previous vertex.  Thus let $p(v_i)=g(v_i)=i$.  Since under the coloring $p$, $\{v_1, v_{m+i}\}$ and $\{v_i, v_{2m+i}\}$ cross edges incident to vertices colored $\{1, 2, \ldots, i\}\cup \{{m+1},\ldots, {m+i-1}\}$  the two endpoints of the path must also be colored differently from all $\{1, 2, \ldots, i\}\cup \{{m+1},\ldots, {m+i-1}\}$.  Since the edges $\{v_i, v_{m+i}\}, \{v_i, v_{2m+i}\}$ do not cross, we may let $p(v_{m+i})=p(v_{2m+i})=m+i$.  However, since the edges $\{v_i, v_{m+i}\}, \{v_i, v_{2m+i}\}$ cross every edge of our $C_3$, they cannot be identified by $g$.  Thus let $g(v_{m+i})=m+i$ and $g(v_{2m+i})=2m+i$.\medskip

In the coloring $g$ we used all $3m$ available colors; thus  $X(\G)=3m$.  While in $p$ we never needed the colors $2m+2, \ldots, 3m$; thus $X'(\G)=3m-(m-1)$.  Therefore $X(\G)-X'(\G) = m-1 = n$, as desired. \end{proof}

\section {Conflict of Interest}

The authors have no conflict of interest to declare.

\bibliographystyle{plain}
\bibliography{GeochromIndep}

\end{document}